\documentclass{article}

\usepackage{arxiv}

\usepackage{bm}
\usepackage[utf8]{inputenc} % allow utf-8 input
\usepackage[T1]{fontenc}    % use 8-bit T1 fonts
\usepackage{hyperref}       % hyperlinks
\usepackage{url}            % simple URL typesetting
\usepackage{booktabs}       % professional-quality tables
\usepackage{amsfonts}       % blackboard math symbols
\usepackage{nicefrac}       % compact symbols for 1/2, etc.
\usepackage{microtype}      % microtypography
\usepackage{lipsum}
\usepackage{graphicx}
\usepackage{color}
\usepackage{dsfont}
\usepackage[authoryear]{natbib}
\newcommand{\diag}{\textrm{Diag}}
\newcommand{\Tr}{\textrm{Tr}}

\usepackage{amsmath,amssymb,mathrsfs,amsthm}
\newtheorem{proposition}{Proposition}[section]
\newtheorem{theoreme}{Theorem}[section]
\newtheorem{corollary}{Corollary}[section]
\newtheorem{assumption}{Assumption}[section]

\newtheorem{property}{Property}[section]

\newtheorem{definition}{Definition}[section]

\title{Optimal Data Splitting for Holdout Cross-Validation in Large Covariance Matrix Estimation}

\author{
    Lamia Lamrani$^{1}$, Christian Bongiorno$^{1}$, Marc Potters$^{2}$ \\ 
    $^{1}$Universit\'e Paris-Saclay, CentraleSup\'elec, \\ 
    Laboratoire de Math\'ematiques et Informatique pour la Complexit\'e et les Syst\`emes, 91192 Gif-sur-Yvette, France \\
    $^{2}$Capital Fund Management, 23 Rue de l'Universit\'e, 75007 Paris, France \\
    \texttt{lamia.lamrani@centralesupelec.fr}
}

\begin{document}

\maketitle

\begin{abstract}
Cross-validation is a statistical tool that can be used to improve large covariance matrix estimation. Although its efficiency is observed in practical applications and a convergence result towards the error of the non linear shrinkage is available in the high-dimensional regime, formal proofs that take into account the finite sample size effects are currently lacking. To carry on analytical analysis, we focus on the holdout method, a single iteration of cross-validation, rather than the traditional $k$-fold approach. We derive a closed-form expression for the expected estimation error when the population matrix follows a white inverse Wishart distribution, and we observe the optimal train-test split scales as the square root of the matrix dimension. For general population matrices, we connected the error to the variance of eigenvalues distribution, but approximations are necessary. In this framework and in the high-dimensional asymptotic regime, both the holdout and $k$-fold cross-validation methods converge to the optimal estimator when the train-test ratio scales with the square root of the matrix dimension which is coherent with the existing theory. 
\end{abstract}

\section{Introduction}

% Covariance cleaning is important
Covariance estimation is a major question in statistics that finds applications in many areas such as physics \citep{taylor2013putting}, neuroscience \citep{beltrachini2013shrinkage,zhuang2020technical}, signal processing \citep{ottersten1998covariance} and finance \citep{laloux1999noise,ledoit2004shrunk, pantaleo2011improved}. A simple answer could be to use the sample estimator; however, it is largely inefficient in high dimensions when the order of the matrix $n$ is comparable to the number of observations $t$ of the data \citep{ledoit2004shrunk}. More generally, it is known that the amount of information lost when the sample estimator $\bm{E}$ is used to approximate the population covariance $\bm{C}$ depends only on $q=\frac{n}{t}$ for any population matrix with Gaussian data in the high-dimension setting \citep{bongiorno2023quantifying}. The estimator coincides with the population when $q=0$, then the larger is $q$ the larger is the error \citep{ledoit2004shrunk, ledoit2022power}.

To address the lack of robustness in the sample covariance estimator, researchers have extensively studied Rotational Invariant Estimators (RIE) \citep{bun2017cleaning,ledoit2011eigenvectors,bun2016rotational}. Today, RIEs represent the most widely used approach for covariance cleaning. In simple terms, a RIE is built by maintaining the eigenvectors of the sample covariance and correcting the sample eigenvalues. The reasoning is that when no information is available on the direction of the eigenvectors of the population covariance matrix, one can assume that each basis of eigenvectors has the same probability of appearing, therefore one keeps the sample eigenvectors and focuses on estimating the eigenvalues such that it minimizes a function of choice such as the Mean Square Error (MSE) \citep{ledoit2011eigenvectors}, the Kullback-Leibler divergence \citep{bongiorno2023optimal} or the Stein's loss \citep{ledoit2022quadratic}.
%In the financial case, the number of assets in a portfolio is typically a few hundred, whereas the fast evolution of financial markets makes it irrelevant to consider data older than a few years \citep{michaud1989markowitz}. 
%%%% Why the sample is higher variance (Stein-Estimators)
% State-of-the-art method for covariance cleaning
In the last decades, many RIEs linked with Random Matrix Theory (RMT) were proposed to clean the noise of the sample estimator (see \citep{bun2017cleaning} and \citep{ledoit2022power} for a review). One of the first RMT methods for cleaning is known as eigenvalue clipping. This method suggests that all the eigenvalues outside the sharp limit $(\lambda_{-}, \lambda_{+})$ of the Marcenko-Pastur (MP) distribution should be kept, while the ones that fall within such limit should be regarded as noise and should be shrunk to a single average value \citep{laloux1999noise, bouchaud2009financial}. However, the limits of the MP distribution depend only on the measured aspect ratio $q$, i.e., the size of the matrix $n$ divided by the number of data points $t$. Notably, the measured aspect ratio is not always equal to the effective ratio that best fits the empirical spectral density to the MP distribution \citep{laloux1999noise,bun2017cleaning}. The other widely used method for cleaning covariance matrices is linear shrinkage. The idea is to linearly shrink the sample eigenvalues towards the identity \citep{ledoit2004shrunk,Potters2020AFC}; although this approach works in many practical applications, it is theoretically justified only under some circumstances such as when the population covariance is a white inverse Wishart \citep{Potters2020AFC}. More recently, the optimal shrinkage correction was analytically derived by Ledoit and P\'ech\'e under certain assumptions, including the $12$th absolute central moment finite for the noise and stationary data \citep{ledoit2011eigenvectors}; however, the computation of this quantity from their equation was not straightforward. Numerical solutions, known as Non-Linear Shrinkage (NLS), were implemented and further developed by Ledoit and Wolf \citep{ledoit2012nonlinear, ledoit2017direct}, making it more practical to apply. The NLS, under the high-dimensional limit, converges to the oracle estimator, which is the ideal estimator that minimizes the MSE with the population covariance matrix. It is worth remarking that the direct estimation of the oracle is unattainable in practice since it relies on the knowledge of the population covariance matrix, which is unknown  \citep{ledoit2011eigenvectors}. 

In practical scenarios, where the data may not be Gaussian or stationary, alternative methods inspired by machine learning find their best regime of applicability. Techniques such as Cross-Validation (CV), some based on factor models \citep{fama1993common,fan2008high}, and others employing Bayesian methods with prior distributions for estimating the covariance matrix \citep{daniels1999nonconjugate} can, in principle, outperform theoretically optimal methods, such as those derived by Ledoit and P\'ech\'e, under these complex scenarios. However, there is no theoretical evidence to support this claim, as deriving analytical results for these methods is highly complex. To fix this gap, the focus of this work is specifically on covariance cleaning through CV.

% CV in statistics
CV is one of the most widely used machine learning methods for model selection and evaluation and belongs to the family of Monte-Carlo methods (see \citep{berrar2019cross} for a review of the main CV procedures). Its applications include a wide range of areas such as ecology \citep{yates2023cross}, medical imaging \citep{bradshaw2023guide}, psychology \citep{de2020cross}, and finance \citep{tan2020estimation,lam2020high}, to quote a few. Part of the reason CV is widely used is because it almost needs no assumptions on the data, it helps reduce the risk of over-fitting for model selection and its efficiency is undoubtedly in practice \citep{berrar2019cross,tan2020estimation}.
There are several CV procedures; however, the core principle of any of those involves dividing the dataset into separate training and testing subsets, also called in-sample and out-of-sample. The model is trained on the training subset and then evaluated on the testing subset. This process helps to assess how well the model generalizes to unseen data and is particularly useful to detect overfitting \citep{moore2001cross}. When overfitting occurs, the model may perform well on the training set but will show a noticeable drop in performance on the testing set. The most popular CV methods are the $k$-fold CV, the Leave-One-Out (LOO) CV, which is a particular case of the former \citep{berrar2019cross}, and the holdout, also called validation or simple cross-validation \citep{berrar2019cross}. The holdout procedure consists of splitting a single time the data set into a train set and a test set. Therefore, this method consists of a single round of a $k$-fold CV. The recommendation in the literature is to use $10$ to $30\%$ of the data points for the test set and the rest for the train set \citep{berrar2019cross}. Differently from the $k$-fold and LOO CV, the holdout can respect the causality of time series if the train and test data are ordered and if the test data is chosen posterior to the train. 
Practitioners should choose their approach based on the system specificity and the estimation strategy they aim for.  For example, the LOO CV is adequate for bias correction but has a very high variance. On the other hand, if stability is important, one can run a $k$-fold procedure with a small k usually between $5$ and $10$, with $k=10$ being the most common one \citep{marcot2021optimal, wong2019reliable, yadav2016analysis}.

% CV in covariance filtering.
Covariance cleaning approaches inspired by CV were proposed in~\citep{abadir2014design,bartz2016advances,Lam2016}. 
Some authors even claim that the $k$-fold CV outperforms the NLS when applied to financial data \citep{bartz2016advances,morstedt2024cross}. 
The flexibility of this numerical estimation allows for simple modifications such as including exponential smoothing in CV \citep{tan2020estimation}, or pre-filtering the data using hierarchical clustering \citep{bongiorno2022statistically}.
In its original formulation, a LOO CV was proposed in \cite{ledoit2012nonlinear}; however, 
\cite{bartz2016advances} and \cite{Lam2016} reasoned that a $k$-fold would improve the estimation. The current procedure consists of dividing the dataset into $k$-folds and finding the average eigenvalue correction of the train eigenvalues to best fit the test set. Analogously to the $k$-fold problem, as $k$ increases, the train set becomes larger, but the test set shrinks, leading to a higher generalization variance. Finding the optimal way to split the data is challenging, especially because the underlying theory is not well understood. In the literature, the recommendation is again to use $k$ between $5$ and $10$ \citep{bartz2016advances, bun2018overlaps,tan2020estimation}. A first attempt to describe the CV covariance cleaning was done by \cite{abadir2014design}, who studied a type of estimator, but only when the matrix dimension is fixed or becomes negligible with respect to the sample size as they tend to infinity. 
More recently, \cite{Lam2016,lam2015supplement} obtained a conservative convergence result for CV estimators, showing that the Frobenius error of the holdout estimator  (named NERCOME) converges to the corresponding error of the NLS.  Under \cite{Lam2016}'s assumptions, the holdout error as a function of the split should exhibit a plateau in the high-dimension regime. However, such a plateau constitutes a rough approximation in the finite-sample case, whereas our result reveals a distinct error minimum. In fact, some authors have raised concerns about the performance of these splits in the high but finite dimensional setting \citep{engel2017overview,tan2025estimation} and \cite{Lam2016} himself recommends to fit the optimal split directly using the data.

% Our contribution
In this paper, we present the first step toward gaining a better understanding of the performance of CV for large but finite covariance matrix estimation using random matrix theory tools. We compute the element-wise expected MSE between the estimator and the population covariance of a single round for a $k$-fold CV, which coincides with the Frobenius error of the holdout estimator. For the general case, the error requires numerical approximations \citep{ledoit2012nonlinear, ledoit2017direct,bun2017cleaning}.  To carry on this estimation analytically, instead, we assume a white inverse Wishart distributed population covariances. Our motivation is to describe analytically a case with satisfies Lam's convegence result assumptions for which we do not obtain the expected plateau for the convergence region in the high but finite dimensional regime. The inverse Wishart ensemble makes it possible to carry on analytical computations as when the datas are Gaussian, both the Bayesian estimator and linear shrinkage coincide with the oracle estimator of \cite{ledoit2012nonlinear} \citep{alvarez2014bayesian,Potters2020AFCbis}.
% sections with the details of the assumptions (maybe not necessary) 
We believe that studying the holdout method in the finite setting is a necessary step to understand the $k$-fold CV better. Furthermore, while the holdout method may be less efficient than $k$-fold CV on stationary data, its capacity to preserve temporal order could make it particularly useful for many practical applications involving non-stationary data \citep{bongiorno2023filtering}. Therefore, in this paper, we focus on the computation of the expected error of the holdout estimator, deriving an analytical expression in the white inverse Wishart case. Furthermore, this result allows us to find the optimal split between the train and test sets that minimizes the Frobenius error, observing that the optimal splits in the high-dimensional regime are not equivalent in the high but finite setting. 

%outline
This paper is organized as follows: in Section \ref{Section 2}, we set the definitions, in Section \ref{Section 3}, we derive analytically the expected Frobenius error of the holdout estimator of for the general case and for a white inverse Wishart population covariance and Gaussian data as well as the optimal split that minimizes the holdout error. Then we compare the error of the holdout estimator with the error of the oracle estimator.

\section{Definitions} \label{Section 2}

\subsection{Multivariate Distributions}

In this section, we introduce the fundamental concepts related to covariance matrix estimation using samples drawn from a multivariate normal distribution. We begin by defining the Wishart distribution and its relationship to the sample covariance matrix, which serves as the maximum likelihood estimator under normality assumptions.

Let $\bm{x} \in \mathbb{R}^{n}$ be a random vector following a multivariate normal distribution with zero mean and population covariance matrix $\bm{\Sigma}$, denoted as $\bm{x} \sim \mathcal{N}(\bm{0}, \bm{\Sigma})$. Let $\bm{x}_1, \dots, \bm{x}_t$ be $t$ independent and identically distributed (i.i.d.) observations of $\bm{x}$. The data matrix $\bm{X} \in \mathbb{R}^{n \times t}$ is formed by concatenating these observations
\begin{equation} \bm{X} = \{ \bm{x}_1 , \bm{x}_2 , \cdots , \bm{x}_t \}. \end{equation} 

Then sample covariance matrix $\bm{E}$ is defined as \begin{equation} \label{eq:sample_cov_gaussian} \bm{E} = \frac{1}{t} \bm{X} \bm{X}^\top. \end{equation}

Under the assumption of normally distributed data, $\bm{E}$ is the maximum likelihood estimator of the population covariance matrix $\bm{\Sigma}$ \citep{muirhead2009aspects}. Moreover, if $\bm{x}_1, \dots, \bm{x}_t$ are i.i.d. samples from $\mathcal{N}(\bm{0}, \bm{\Sigma})$, then the sample covariance matrix $\bm{E}$ follows a scaled Wishart distribution: \begin{equation}  \bm{E} \sim \frac{1}{t} \mathcal{W}_n(t, \bm{\Sigma}), \end{equation} where $\mathcal{W}_n(t, \bm{\Sigma})$ denotes the Wishart distribution with $t$ degrees of freedom and scale matrix $\bm{\Sigma}$.

In this paper, we will mostly consider the high-dimensional limit which refers to the asymptotic regime where both $n$ and $t$ tend to infinity while the aspect ratio $q$ remains fixed \begin{equation}\label{def:high_dim_limit} n, t \to \infty \quad \text{with} \quad q = \frac{n}{t} \quad \text{fixed}. \end{equation}

Next, we introduce the inverse Wishart distribution, which is the distribution of the inverse of a Wishart-distributed matrix. We define the white inverse Wishart distribution and its scaled version, which will be used as the model for the population covariance matrix in our analysis.

\begin{definition} Let $\bm{W} \sim \mathcal{W}_n(t^*, \bm{\Psi})$ be a Wishart-distributed random matrix with $t^*$ degrees of freedom and scale matrix $\bm{\Psi}$. If $t^* > n - 1$, the inverse of $\bm{W}$ follows an inverse Wishart distribution, denoted as \begin{equation} \bm{\Sigma} = \bm{W}^{-1} \sim \textrm{Inv-}\mathcal{W}_n(t^*, \bm{\Psi}^{-1}). \end{equation} 

A white inverse Wishart distribution is an inverse Wishart distribution where the scale matrix is proportional to the identity matrix, i.e., $\bm{\Psi} = \psi \bm{\mathds{1}}$ for some scalar $\psi > 0$. In this case, the inverse Wishart-distributed matrix $\bm{\Sigma}\sim \textrm{Inv-}\mathcal{W}_n^{-1}(t^*, \psi^{-1} \bm{\mathds{1}})$ is called white because its expected value is proportional to the identity matrix. \end{definition}

To ensure that the expected value of $\bm{\Sigma}$ is the identity matrix, we choose the scale parameter appropriately.

Then, for $\bm{\Sigma} \sim \mathcal{W}_n^{-1}(t^*, (t^* - n - 1) \bm{\mathds{1}})$ with $t^* > n + 1$, we obtain that\begin{equation}
\label{eq:exp_trace_invW}
\mathbb{E}[\bm{\Sigma}] = \bm{\mathds{1}}. \end{equation}

Indeed, the expected value of an inverse Wishart-distributed matrix is given by \begin{equation} \mathbb{E}[\bm{\Sigma}] = \frac{\bm{\Psi}}{t^* - n - 1}. \end{equation} By setting $\bm{\Psi} = (t^* - n - 1) \bm{\mathds{1}}$, we obtain $\mathbb{E}[\bm{\Sigma}] = \bm{\mathds{1}}$.

This choice of scaling ensures that the population covariance matrix $\bm{\Sigma}$ has an expected value equal to the identity matrix, making it a convenient model for theoretical analysis.

To simplify the notation we define the following parametrization of the white inverse Wishart.
\begin{definition}
\label{def:white_inv}
Let $q^*=n/t^*$ and $p=q^*/(1-q^*)$, then we refer to
\begin{equation}
    \bm{\Sigma} \sim \mathcal{W}^{-1}_{np} =  \textrm{Inv-}\mathcal{W}_{n}(t^*,(t^*-n-1) \mathds{1})
\end{equation}
\end{definition}

In this work, we will use the white inverse Wishart distribution as the model for the population covariance matrix $\bm{\Sigma}$ of multivariate normal distributed data defined in Eq.\ \eqref{eq:sample_cov_gaussian}.

\subsubsection{Estimation of Parameters}

We now discuss how to estimate parameters related to the covariance matrices, particularly focusing on the trace and the normalized trace of their powers.

\begin{definition} The normalized trace of a matrix $\bm{A} \in \mathbb{R}^{n \times n}$ is defined as \begin{equation} \tau(\bm{A}) = \frac{1}{n} \Tr(\bm{A}) \end{equation} 
where $\Tr(\bullet)$ in the trace of a matrix.

\end{definition}

If $\bm{\Sigma} \sim \mathcal{W}^{-1}_{np}$ is a matrix extracted from a white inverse Wishart with parameters $\{n,p\}$ then, in the high-dimensional limit of definition ~\eqref{def:high_dim_limit} with $p$ negligible compared to $n$, we have

\begin{equation}
\label{eq:sigma_square}\mathbb{E}\left[\tau(\bm{\Sigma}^2)\right] = 1+p.
\end{equation}

To derive the normalized trace of an inverse Wishart matrix, let us use that the variance of the elements of $\bm{\Sigma}$ is \citep{haff1979identity}

\begin{equation}\label{eq:var_invwishart}
    \mathbb{V}\left[\Sigma_{ij}\right] = \frac{p (n-p) +  p (n+p) \delta_{ij}}{n (n - 3p)},
\end{equation}
Therefore, by using the expected value of the normalized trace of $\bm{\Sigma}$ written in eq.\ \eqref{eq:exp_trace_invW}, we can derive the 
\begin{equation}
    \mathbb{E}\left[\tau(\bm{\Sigma}^2)\right] = \frac{1}{n}\sum_{ij} \left( \mathbb{V}\left[\Sigma_{ij}\right] +  \mathbb{E}\left[\Sigma_{ij}\right]^2\right) = \frac{(n-p)(n p-p + n)}{n ( n- 3 p)}.\label{eq:tsq}
\end{equation}
In the high-dimensional limit, defined in definition ~\eqref{def:high_dim_limit}, the former equation simplifies to the known formula \citep{Potters2020AFC}  
\begin{equation}
    \mathbb{E}\left[\tau(\bm{\Sigma}^2)\right] = \lim_{n \to \infty} \frac{(n-p)(n p-p + n)}{n ( n- 3 p)} =  1+p.
\end{equation}

Let us note that from equation ~\eqref{eq:var_invwishart}, $p<n/3$ is a necessary condition to have a finite variance.

To estimate $p$ from a sample covariance $\bm{E}$ with population covariance $\bm{\Sigma}$ we can use the expected value of the trace of $\bm{E}^{2}$. Let $\bm{E} = \frac{1}{t} \bm{X} \bm{X}^\top$ be the sample covariance matrix based on $t$ i.i.d. samples from $\mathcal{N}(\bm{0}, \bm{\Sigma})$, where $\bm{\Sigma}$ is the population covariance matrix.  Then, in the high dimensional limit, 
\begin{equation}  
\label{eq:trace_samplecov_square}
\mathbb{E}\left[ \tau(\bm{E}^2) \right] = \mathbb{E}\left[ \tau(\bm{\Sigma}^2) \right] + q, \end{equation} where $q = n/t$ is the aspect ratio of the data matrix \citep{Potters2020AFC}. 

\subsection{Covariance matrix estimation}

In this section, we define the theoretical background for two of the most popular families of covariance matrix estimators, the RIE and the CV estimators.

\subsubsection{Oracle estimator}

\begin{definition}
$\bm{\Xi}(\bm{E})$ is a RIE of the population covariance matrix $\bm{\Sigma}$, where $\bm{E}$ denote the sample covariance over the data. This estimator satisfies the following distributional equality for all orthogonal matrices $\bm{O}$\citep{bun2016rotational}
\begin{equation}
 \bm{\Xi}(\bm{O}\bm{E}\bm{O}^{T})\overset{\textrm{distr}}{=}\bm{O}\Xi(\bm{E}) \bm{O}^{T}.   
\end{equation}
\end{definition}

The Frobenius distance is often used to evaluate the performance of a matrix estimator as it corresponds to the element-wise mean-squared error between the estimator and the population covariance.

\begin{definition}
    The normalized Frobenius estimation error between an estimator $\bm{\Xi}$ and a population matrix $\bm{\Sigma}$ is \citep{Potters2020AFC}
    \begin{equation}
    \Vert \bm{\Xi} - \bm{\Sigma}\Vert^2_{\textrm{F}}=\tau ((\bm{\Xi} - \bm{\Sigma})^2).
    \end{equation}
\end{definition}

Let us mention that the expected Frobenius error of the sample covariance Wishart distributed is known and does not depend on the distribution of the population covariance \citep{Potters2020AFC}.For $\bm{E}$ a sample covariance matrix from multivariate normal data $X \in \mathbb{R}^{n \times t}$ with population covariance $\bm{\Sigma}$, in the high-dimensional limit, the expected Frobenius error between $\bm{E}$ and $\bm{\Sigma}$ is equal to 

\begin{equation}
\label{eq:err_sample_cov}
\mathbb{E}\left[\tau((\bm{E}-\bm{\Sigma})^2)\right]=q.
\end{equation}

In order to define the oracle estimator of the covariance matrix, let us first recall that $\bm{E}$ can be diagonalized on an orthogonal basis. Its eigenvectors are called the sample eigenvectors and its eigenvalues are contained in the diagonal matrix $\bm{\Lambda}$ and are non-negative
\begin{equation}
    \bm{E}=\bm{V\Lambda V}^{T},
\end{equation}
with the columns of $\bm{V}$ arranged so that the associated eigenvalues in $\bm{\Lambda}$ increase sequentially.

Let us now define the oracle estimator.
\begin{definition}
The RIE which minimizes the Frobenius error with the population covariance $\bm{\Sigma}$ is called the oracle estimator and is equal to \citep{ledoit2011eigenvectors}

\begin{equation}
    \bm{\Xi}^{O}(\bm{E},\bm{\Sigma})=\bm{V}\diag(\bm{V}^{T}\bm{\Sigma V})\bm{V}^{T},
\end{equation}

where $\bm{V}$ are the sample eigenvectors of $\bm{E}$ and $\diag(\bullet)$ is the operator that sets the out-of-diagonal entries to zero.

The oracle eigenvalues are then
\begin{equation}\label{eq:oracle}
    \bm{\Lambda}^O(\bm{E},\bm{\Sigma})=\diag(\bm{V}^{T}\bm{\Sigma V}).
\end{equation}

\end{definition}

This estimator is called the ``oracle'' estimator because it still depends on the unknown population covariance matrix that we want to estimate. Ledoit and P\'ech\'e derived an estimator that depends only on observable quantities and converges to the oracle in the high-dimensional limit \citep{ledoit2011eigenvectors}.
 
\begin{theoreme}
Let us call $\bm{E}$ a sample covariance of size $n\times n$ generated by $t$ i.i.d. observations with covariance population matrix $\bm{\Sigma}$. Let us assume that the noise used to generate $\bm{E}$ is independent of $\bm{\Sigma}$, has a finite $12$th absolute central moment and that the distribution of $\bm{\Sigma}$ converges for $n$ tends to $\infty$ at every point of continuity to a non-random distribution which defines a probability measure and whose support is included in a compact interval.  Let $q=n/t$ and $g_{\bm{E}}$ the Stieltjes transform of $\bm{E}$. For $\lambda$, an eigenvalue of $\bm{E}$, let $\xi_{\lambda}(\bm{E})$ be the corresponding eigenvalue of the oracle's estimator. Then, in the high-dimensional limit,  for $z=\lambda+i\eta$ with $\lambda\in Sp(\bm{E})$ and $\eta>0$ \citep{ledoit2011eigenvectors}
\begin{equation}
\label{ledoit_peche}
\xi_{\lambda}(\bm{E})=\lim\limits_{\eta \rightarrow 0^{+}} \frac{\lambda}{\vert 1-q+q\lambda g_{\bm{E}}(\lambda+i\eta)\vert^2}.
\end{equation}
\begin{flushright}
(Ledoit-P\'ech\'e formula)
\end{flushright}
\end{theoreme}

\cite{ledoit2012nonlinear} claim that, even when the data matrices are finite but large, the Ledoit-P\'ech\'e eigenvalues still provide a good approximation of the oracle eigenvalues. 

In the high-dimensional limit for a white inverse Wishart population covariance, the Ledoit-P\'ech\'e formula is a linear shrinkage that coincides with the oracle estimator. It is worth mentioning that the Bayesian estimator converges to the same formula in the high-dimensional limit \citep{Potters2020AFCbis}.  In such conditions, the expected Frobenius error between the estimator and the population covariance can be derived.
\begin{property}
\label{prop:oracle_invW_gaussian_error}
    For $\bm{\Sigma}$, an inverse Wishart population covariance matrix of parameters $n$ and $p$ and $\bm{E}$ the sample covariance of $\bm{\Sigma}$ generated on $t$ i.i.d. centered and Gaussian observations. Then, in the high-dimensional limit, the expected Frobenius error between the oracle estimator and $\bm{\Sigma}$ is \citep{Potters2020AFCbis, llamrani}
\begin{equation}
\label{er_lp}
   \mathbb{E}\left[\tau((\bm{\Xi}^{\textrm{O}}(\bm{E},\bm{\Sigma})-\bm{\Sigma})^2)\right] =\frac{pq}{p+q}.
\end{equation}
\end{property}

An important remark is that the expected Frobenius error of the oracle estimator is smaller than the one from the sample covariance written in equation~\eqref{eq:err_sample_cov} for any $q\neq0$ even though the sample estimator is the maximum likelihood estimator.

\subsubsection{Cross-validation estimators}

In this section, we define the two main CV methods: the holdout method and the $k$-fold CV. Additionally, we introduce two variations of these methods, one with and one without rotational invariance \citep{abadir2014design,Lam2016}. To explain these methods, we need first to define the index partitions.

\begin{definition}\label{def:partition}
    Let $t$ be the total number of observations, and let $t_\textrm{out}$ be an integer between $1$ and $t-1$. Define  $t_\textrm{in} = t - t_\textrm{out}$. We partition the set of indices $\mathcal{I} = \{1, 2, \dots, t\}$ into $2$ disjoint subsets $\mathcal{I}_\textrm{out}$ and $\mathcal{I}_\textrm{in}$ of size $t_{\textrm{out}}$ and $t_{\textrm{in}}$ respectively, such that

\begin{equation}
    \mathcal{I}_\textrm{in} \cup \mathcal{I}_\textrm{out} = \mathcal{I}, \quad   \mathcal{I}_\textrm{in} \cap \mathcal{I}_\textrm{out}  =\emptyset
\end{equation}
\begin{flushright}
(index partition)
\end{flushright}
\end{definition}
\begin{definition}\label{def:kpartition}
    Let $t $ be the total number of observations, and let $t_{\textrm{out}}$ be an integer between $1$ and $t-1$ such that $t_{\textrm{out}}$ divides $t$. Define $k = t/t_{\textrm{out}}$ and $t_{\textrm{in}} = t - t_{\textrm{out}}$. For the k-fold, we partition the set of indices $\mathcal{I} = \{1, 2, \dots, t\}$ into $k$ disjoint subsets $\{\mathcal{I}_{\textrm{out}, l}\}_{l=1}^{k}$, each of size $t_{\textrm{out}}$, such that

\begin{equation}
\bigcup_{l=1}^{k} \mathcal{I}_{\textrm{out}, l} = \mathcal{I}, \quad \mathcal{I}_{\textrm{out}, l} \cap \mathcal{I}_{\textrm{out}, m} = \emptyset \text{ for } l \neq m.
\end{equation}
For each $l$, the training set indices are defined as
\begin{equation}
    \mathcal{I}_{\textrm{in}, l} = \mathcal{I} \setminus \mathcal{I}_{\textrm{out}, l}.
\end{equation} 
\begin{flushright}
($k$-fold index partition)
\end{flushright}

\end{definition}

Now, let us define two versions of the holdout estimator.
\begin{definition}
    Let us consider $\bm{X}=\{\bm{x}_{1}, ..., \bm{x}_{t}\}$ an i.i.d. set of centered observations of covariance $\bm{\Sigma}$.
    Let $t_\textrm{out}$ an integer between $1$ and $t-1$, and  $\{\mathcal{I}_\textrm{in},\mathcal{I}_\textrm{out}\}$ an index partition (definition ~\eqref{def:partition}) defined from $t_\textrm{out}$.  Let us call $\bm{E}_\textrm{in}$ the sample covariance over the subset of observations $\bm{X}_\textrm{in}=\{\bm{x}_i|\, i \in \mathcal{I}_\textrm{in}\}$ (train set) and $\bm{E}_\textrm{out}$ the sample covariance on the remaining observations $\bm{X}_\textrm{out}=\{\bm{x}_i|\, i \in \mathcal{I}_\textrm{out}\}$ (test set). Let us call $\bm{V}_\textrm{in}$ the eigenvectors of $\bm{E}_\textrm{in}$ ranked by increasing eigenvalues. Then the holdout estimator is \citep{Lam2016}
    
\begin{equation}
\label{holdout_def}
\bm{\Xi}^\textrm{H} := \bm{\Xi}^O(\bm{E}_\textrm{in},\bm{E}_\textrm{out})=\bm{V}_\textrm{in}\diag(\bm{V}_\textrm{in}^{T}\bm{E}_\textrm{out}\bm{V}_\textrm{in})\bm{V}_\textrm{in}^{T}.
    \end{equation}

    \begin{flushright}
        (holdout estimator)
    \end{flushright}
\end{definition}

The formerly defined estimator is not a RIE, in fact, it does not use the sample eigenvectors. Instead, from the following definition, we can define a RIE  

\begin{definition}
Let us consider $\bm{X}=\{\bm{x}_{1}, ..., \bm{x}_{t}\}$ an i.i.d. set of centered observations of population covariance $\bm{\Sigma}$ and sample covariance $\bm{E}$. Let $t_\textrm{out}$ an integer between $1$ and $t-1$,and  $\{\mathcal{I}_\textrm{in},\mathcal{I}_\textrm{out}\}$ an index partition (definition ~\eqref{def:partition}) defined from $t_\textrm{out}$.  Let us call $\bm{E}_\textrm{in}$ the sample covariance over the subset of observations $\bm{X}_\textrm{in}=\{\bm{x}_i|\, i \in \mathcal{I}_\textrm{in}\}$ (train set) and $\bm{E}_\textrm{out}$ the sample covariance on the remaining observations $\bm{X}_\textrm{out}=\{\bm{x}_i|\, i \in \mathcal{I}_\textrm{out}\}$ (test set). Let us call $\bm{V}_\textrm{in}$ and $\bm{V}$ the eigenvectors of $\bm{E}_\textrm{in}$ and  $\bm{E}$ respectively ranked according to their eigenvalues. Then the holdout eigenvalues are \citep{abadir2014design}
\begin{equation}    
\bm{\Lambda}^H(\bm{E}_\textrm{in},\bm{E}_\textrm{out})  = \diag(\bm{V_\textrm{in}}^{T}\bm{E}_\textrm{out}\bm{V}_\textrm{in}),
    \end{equation}
    \begin{flushright}    
        (\textit{holdout eigenvalues})
    \end{flushright}
and the rotational invariant holdout estimator is
\begin{equation}    
 \bm{\Xi}^{RH} :=  \bm{V} \bm{\Lambda}^H(\bm{E}_\textrm{in},\bm{E}_\textrm{out}) \bm{V}^{T}.
    \end{equation}
    \begin{flushright}    
        (\textit{rotational invariant holdout estimator})
    \end{flushright}
\end{definition}

Analogously, we present two versions of the $k$-fold CV.
%%% we should talk about  $k$-fold
\begin{definition}
    Let us consider $\bm{X}=\{\bm{x}_{1}, ..., \bm{x}_{t}\}$ an i.i.d. set of centered observations of covariance $\bm{\Sigma}$. Let $t_\textrm{out}$ be an integer between $1$ and $t-1$ such that $t_\textrm{out}$ divides $t$, and let $\{\mathcal{I}_{\textrm{out}, l}, \mathcal{I}_{\textrm{in}, l}\}$ be a $k$-fold index partition (definition~\eqref{def:kpartition}). For each $1 \leq l \leq k$, let $\bm{E}_{\textrm{in}, l}$ be the sample covariance over the observations $\bm{X}_{\textrm{in}, l} = \{ \bm{x}_i | i \in \mathcal{I}_{\textrm{in}, l} \}$ (train set) and $\bm{E}_{\textrm{out}, l}$ the sample covariance over the observations $\bm{X}_{\textrm{out}, l} = \{ \bm{x}_i | i \in \mathcal{I}_{\textrm{out}, l} \}$ (test set). Let $\bm{V}_{\textrm{in}, l}$ the eigenvectors of $\bm{E}_{\textrm{in}, l}$ ranked by increasing eigenvalues. Then the $k$-fold CV estimator is \citep{Lam2016}
\begin{equation}
\label{k_fold_cv_def}
\bm{\Xi}^\textrm{CV} = \frac{1}{k} \sum_{l=1}^{k} \bm{\Xi}_l^\textrm{H} = \frac{1}{k} \sum_{l=1}^k \bm{V}_{\textrm{in}, l}\diag(\bm{V}_{\textrm{in}, l}^{T}\bm{E}_{\textrm{out}, l}\bm{V}_{\textrm{in}, l})\bm{V}_{\textrm{in}, l}^{T}.
\end{equation}

\begin{flushright}
    ($k$-fold CV estimator)
\end{flushright}

\end{definition}

\begin{definition}
Let us consider $\bm{X}=\{\bm{x}_{1}, ..., \bm{x}_{t}\}$ an i.i.d. set of centered observations of population covariance $\bm{\Sigma}$ and sample covariance $\bm{E}$. Let $t_\textrm{out}$ be an integer between $1$ and $t-1$ such that $t_\textrm{out}$ divides $t$, and let $\{\mathcal{I}_{\textrm{in}, l},\mathcal{I}_{\textrm{out}, l}\}_{l=1}^{k}$ be a $k$-fold index partition (definition~\eqref{def:kpartition}). For each $1 \leq l \leq k$, let $\bm{E}_{\textrm{in}, l}$ be the sample covariance over the observations $\bm{X}_{\textrm{in}, l} = \{ \bm{x}_i | i \in \mathcal{I}_{\textrm{in}, l} \}$ (train set) and $\bm{E}_{\textrm{out}, l}$ the sample covariance over the observations $\bm{X}_{\textrm{out}, l} = \{ \bm{x}_i | i \in \mathcal{I}_{\textrm{out}, l} \}$ (test set). Let $\bm{V}_{\textrm{in}, l}$ and $\bm{V}$ be eigenvectors of $\bm{E}_{\textrm{in}, l}$ and $\bm{E}$ respectively ranked according to their eigenvalues. Then the $k$-fold CV eigenvalues are \citep{abadir2014design}
\begin{equation}
\bm{\Lambda}^{\textrm{CV}}=\frac{1}{k}\sum_{l=1}^k \bm{\Lambda}^H_l = \frac{1}{k}\sum_{l=1}^k \diag(\bm{V}_{\textrm{in}, l}^T \bm{E}_{\textrm{out}, l} \bm{V}_{\textrm{in}, l}).
\end{equation}

The rotational invariant $k$-fold CV estimator is then
\begin{equation}
    \bm{\Xi}^{\textrm{RCV}} = 
    \bm{V} \bm{\Lambda}^{\textrm{CV}} \bm{V}^T.
\end{equation}

\begin{flushright}
    (rotational invariant $k$-fold CV)
\end{flushright}

\end{definition}

We observed numerically that the first version of CV outperforms the second one; however, in this case, we do not have analytical proof. 
%%%% Discuss the order of the indices and cite that if there is a meaningfull order CV might be problematic

In the general case, the specific partitioning of the train and test indices can influence the performance of the estimator, unlike the i.i.d. setting, where only the size of the train and test sets matter.    

Before considering the computation of the expected error of estimation of the holdout estimator, let us first state the assumptions of the convergence result of \cite{Lam2016,lam2015supplement} for the holdout error towards the NLS error.
\begin{assumption}\label{assumptions_lam}\item[(H1)] The observations can be written as $\bm{X}=\sqrt{\bm{\Sigma_{n}}}\bm{Y}$ where $\bm{Y}$ is a $n\times t$ matrix of real or complex random variables with zero mean and unit variance such that $\mathbb{E}\left(\vert Y_{ij} \vert^{l} \right)\leq b<\infty$ for some constant $b$ and for $2<l\leq 20$. 
        \item[(H2)] The population covariance matrix $\bm{\Sigma}_{n}$ is non-random, $\Vert \bm{\Sigma}\Vert_{L^{2}}=\mathcal{O}(1)$ and $\bm{\Sigma}\neq \sigma^{2}\bm{\mathds{1}}$
    \item [(H3)] $\frac{n}{t}\to q>0$ as $n\to\infty$

    \item [(H4)] Let us consider $(\tau_{1}\dots \tau_{n})$ the eigenvalues of $\bm{\Sigma}_{n}$, the empirical spectral distribution of $\bm{\Sigma}_{n}$
    \begin{equation}
        H_{n}(\tau)=\frac{1}{n}\sum_{j=1}^{n}\mathds{1}_{[\tau_{j},+\infty[}(\tau)
    \end{equation}
    converges to a non-random limit $H(\tau)$ at every point of continuity of $H$. Moreover, H defines a probability distribution function whose support $Supp(H)$ is included in a compact interval $[h_{1},h_{2}]$ where $0<h_{1}\leq h_{2}<+\infty$.    
\end{assumption}

Let us now state \cite{Lam2016}'s convergence result in the asymptotic regime

\begin{theoreme}
\label{theorem_lam}
    Under the assumptions listed in \ref{assumptions_lam}, the holdout estimator error converges towards the error of the NLS, with respect to the Frobenius error and Stein's loss provided that the split satisfies $\frac{t_{\textrm{in}}}{t}\xrightarrow[t \to \infty]{} 1$, $t_{\textrm{out}} \xrightarrow[t \to \infty]{} \infty$ and
\begin{equation}
\label{eq:lam_condition}
   \sum_{t\geq 1} \frac{n} {t_{\textrm{out}}^{5}}<\infty. 
\end{equation}
\end{theoreme}

For example, $k=\alpha\sqrt{n}$ would imply $t_{\textrm{out}}$ proportional to the square root of $n$ and would satisfy the condition on the split written in eq.\ \eqref{eq:lam_condition} taking into account that $n=q t$ in the high-dimensional regime.

The last two assumptions are included in the assumptions of the Ledoit and Péché's formula \citep{ledoit2011eigenvectors}, moreover our setting of Gaussian data with an inverse Wishart population covariance matrix satisfies all of Lam's assumptions except that we did not exclude the case where the population covariance matrix is proportional to the identity. Therefore, we expect to observe the asymptotic convergence of the holdout Frobenius error towards the one of the NLS.

\section{Derivation of the Estimation Error for the Holdout}
\label{Section 3}

In the forthcoming section, we compute analytically the expected Frobenius error of the holdout estimator for Gaussian i.i.d. data. Since the general case is not fully analytical, we derived a closed form of the Frobenius estimator error when the population covariance is drawn from a white inverse Wishart distribution. In this regime, we can also identify the optimal split. 

Before carrying on the derivation of the error, it is worth having a look at the shapes of the Frobenius error as a function of the as a function of the train-test ratio factor $k=t/t_\textrm{out}$. In Fig. \ref{fig:sim_cv_holdout_errors_log_scale}, we show that either CV and Holdout methods exhibit a minimum error point that is not at an extreme value of $k$. Although the exact locations of the minimum can be different among the two methods, numerical simulations seem to indicate a similar scale factor. 
This discrepancy might arise from the distinct meaning of $k$ in each case: in the CV method, $k$ represents the number of splits, while in the Holdout method, it reflects only the imbalance between the sizes of the training and testing sets.

\begin{figure}[tbh]
        \includegraphics[width=0.495\columnwidth]{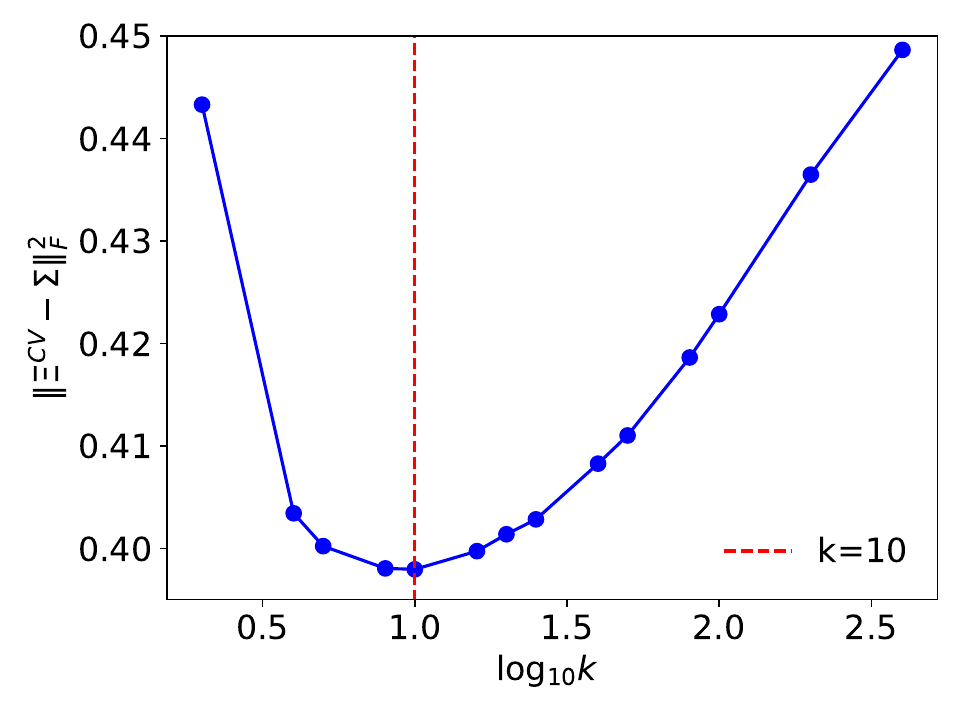}
        \includegraphics[width=0.495\columnwidth]{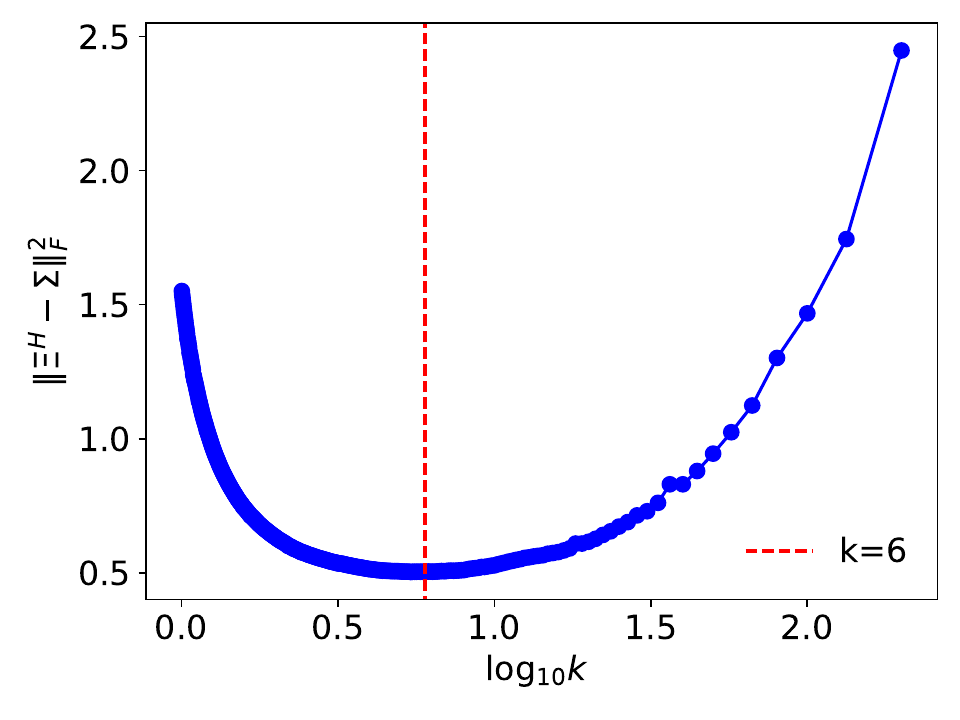}  
    \caption{The left panel shows the average Frobenius error between the population covariance and the CV estimator of equation~\eqref{k_fold_cv_def} as a function of the log of the $k$-fold. The right panel shows the average Frobenius error between the population covariance and the Holdout estimator of equation~\eqref{holdout_def} has a function of the log of the train-test ratio factor $k$, integers which divides $t$ the number of data points. The averages are computed with $100$ simulations of inverse Wishart population matrices with $n=200$ and $p=1.5$, with Gaussian data with $q=0.5$}\label{fig:sim_cv_holdout_errors_log_scale}
\end{figure}

\subsection{General case}

The proof involves two steps. The first one is to apply Wick's theorem to link the moments of the sample covariances to the population covariance matrix using the Gaussian assumption \citep{wick1950evaluation}. The second step is to recognize the oracle estimator eigenvalues for which we have a simple expression in the inverse Wishart case using the Ledoit-P\'ech\'e's formula \citep{Potters2020AFC}.

Let us now announce the main result of this paper.

\begin{proposition}
\label{prop1}
 Let us consider $\bm{\Sigma}$ a positive-defined matrix in $\mathbb{R}^{n \times n}$ and $(\bm{x}_{\textrm{i}})_{1\leq i\leq t}$ i.i.d. random vectors of $\mathbb{R}^{n}$ such that $\bm{x}_{i}\sim\mathcal{N}(0,\bm{\Sigma})$. Let $\bm{\Xi}^\textrm{H}$ the holdout estimator associated with the index partition $\{\mathcal{I}_\textrm{in},\mathcal{I}_\textrm{out}\}$ defined from $t_\textrm{out}$ integer between $1$ and $t-1$, with $\bm{V}_\textrm{in}$ the eigenvector of the sample covariance $\bm{E}_\textrm{in}$ on the train partition $\mathcal{I}_\textrm{in}$.  Then the expected Frobenius error between the holdout estimator and the population covariance $\bm{\Sigma}$ is

\begin{equation}
\mathbb{E}\left[\tau((\bm{\Xi}^\textrm{H}-\bm{\Sigma})^2)\right]=\left(\frac2{t_{\textrm{out} }}-1\right)\mathbb{E}\left[\tau(\diag(\bm{V}_\textrm{in}^{T}\bm{\Sigma} \bm{V}_\textrm{in})^2)\right]+\mathbb{E}\left[\tau(\bm{\Sigma}^2)\right].
\end{equation}
\end{proposition}

\label{val_proof1}

\begin{proof}
Let us now prove the proposition by deriving the expression of the Frobenius error between the holdout estimator in equation~\ref{holdout_def} and the population matrix $\bm{\Sigma}$. The Frobenius error can be expanded as

\begin{equation}
\label{eq:frob_expansion}
\tau((\bm{\Xi}^\textrm{H}-\bm{\Sigma})^2)=\tau((\bm{\Xi}^\textrm{H})^2)-2\tau(\bm{\Xi}^\textrm{H}\bm{\Sigma})+\tau(\bm{\Sigma}^2).
\end{equation}

The first term can be expressed in terms of in-sample eigenvectors $\bm{V}_\textrm{in}$ and sample covariance $\bm{E}_\textrm{out}$ of the test partition $\mathcal{I}_\textrm{out}$. Given the independence by construction between these two quantities, the equation simplifies to

\begin{equation}
\label{eq_val1_xi_carre}
\mathbb{E}\left[\tau((\bm{\Xi}^\textrm{H})^2)\right]=\frac{1}{n}\sum_{i,j,k,l,m} \mathbb{E}\left[(V_\textrm{in})_{ji}(V_\textrm{in})_{ki}(V_\textrm{in})_{li}(V_\textrm{in})_{mi}\right]\mathbb{E}\left[(E_\textrm{out})_{jk}(E_\textrm{out})_{lm}\right].
\end{equation}

As our data are generated by a Gaussian multiplicative noise, Wick's theorem can be used to compute the moments of the out-of-sample covariance $\bm{E}_\textrm{out}$. Let us first remind the result.

\begin{theoreme}
    Let us consider $\{\bm{x}_{1},\bm{x}_2,..,\bm{x}_{r}\}$ i.i.d. centered Gaussian vectors of $\mathbb{R}^{n}$, then \citep{wick1950evaluation}
    
    \begin{equation}
    \mathbb{E}\left[ \bm{x}_{1}\bm{x}_2...\bm{x}_{r}\right]=\sum_{p\in\mathcal{P}_{r}^2}\prod_{\{i,j\}} \mathbb{E}\left[ \bm{x}_{i}\bm{x}_{j}\right]=\sum_{p\in\mathcal{P}_{r}^2}\prod_{\{i,j\}} Cov( \bm{x}_{i},\bm{x}_{j}).
    \end{equation}

    \begin{flushright}
        (Wick's Formula)
    \end{flushright}
\end{theoreme}
By applying the Wick's theorem, we obtain

\begin{equation}
\mathbb{E}\left[(E_\textrm{out})_{jk}(E_\textrm{out})_{lm}\right]=\frac{1}{(t_\textrm{out})^2}\sum_{n_{1},n_2}^{t_{\textrm{out} }} \mathbb{E}\left[ \Sigma_{jk}\Sigma_{lm} + \delta_{n_{1}n_2} \left(\Sigma_{jm}\Sigma_{kl} + \Sigma_{jl}\Sigma_{km} \right) \right],
\end{equation}
and by re-injecting the previous equation in equation \eqref{eq_val1_xi_carre} 
\begin{equation}
\mathbb{E}\left[\tau((\bm{\Xi}^\textrm{H})^2)\right]=\frac{1}{n}\sum_{i,j,k,l,m} \mathbb{E}\left[(V_\textrm{in})_{ji}(V_\textrm{in})_{ki}(V_\textrm{in})_{li}(V_\textrm{in})_{mi}\right]\left(1+\frac2{t_\textrm{out}}\right)\mathbb{E}\left[\Sigma_{jk}\Sigma_{lm}\right].
\end{equation}

Therefore,
\begin{equation}
\label{eq:app_wick1}
\mathbb{E}\left[\tau((\bm{\Xi}^\textrm{H})^2)\right]=\left(1+\frac2{t_\textrm{out}}\right)\mathbb{E}\left[\tau(\diag(\bm{V}_\textrm{in}^{T}\bm{\Sigma} \bm{V}_\textrm{in})^2)\right].
\end{equation}

By applying again Wick's theorem, we can also obtain
\begin{equation} \label{eq:app_wick2}
    \mathbb{E}\left[\tau(\bm{\Xi}^\textrm{H} \bm{\Sigma})\right]=\mathbb{E}\left[\tau(\diag(\bm{V}_\textrm{in}^{T}\bm{\Sigma} \bm{V}_\textrm{in})^2)\right].
\end{equation}

Then by combining equation~\eqref{eq:app_wick1} and equation~\eqref{eq:app_wick2}, the holdout Frobenius error can be re-written
\begin{align}
\label{eq:after_wick}
\mathbb{E}\left[\tau((\bm{\Xi}^\textrm{H}-\bm{\Sigma})^2)\right]&=\left(\frac2{t_{\textrm{out} }}-1\right)\mathbb{E}\left[\tau(\diag(\bm{V}_\textrm{in}^{T}\bm{\Sigma} \bm{V}_\textrm{in})^2)\right]+\mathbb{E}\left[\tau(\bm{\Sigma}^2)\right]  \\
&= \left(\frac2{t_{\textrm{out} }}-1\right)\mathbb{E}\left[\tau(\bm{\Lambda}^O(\bm{V}_\textrm{in},\bm{\Sigma})^2))\right]+\mathbb{E}\left[\tau(\bm{\Sigma}^2)\right] .
\end{align}

\end{proof}

 From this expression, we recognize the oracle eigenvalues $\bm{\Lambda}^{O}$ of equation~\eqref{eq:oracle}. In the high dimension limit, we can use the Ledoit P\'ech\'e formula (equation~\ref{ledoit_peche}) for $\bm{\Lambda}^{O}$. The formula is exact in the high dimension limit but remains a very good approximation when the dimensions, $n$ and $t$ are finite but large \citep{ledoit2012nonlinear}. 
 The formula in equation~\eqref{eq:after_wick} is the general holdout error for Gaussian multiplicative noise and establishes a quantitative link with the oracle estimator eigenvalues and it is not restricted to the inverse Wishart case. One can compute the holdout error for any population covariance matrix as long as one can apply the Ledoit-P\'ech\'e formula written in equation~\eqref{ledoit_peche}. Therefore when $\bm{\Sigma}$ has a spectral density which converges a.s. at every point of continuity to a non-random limit when $n$ tends to $\infty$ and so that the support of the limiting distribution is included in a compact interval \citep{ledoit2011eigenvectors}.  Moreover, this equation can be solved numerically for any population covariance verifying these regularity conditions \citep{ledoit2012nonlinear, ledoit2017direct}. 
 
It is interesting to observe that both traces can be expressed in terms of expectation and variance of the eigenvalue distributions, in fact
\begin{equation}
    \mathbb{E}\left[\tau(\bm{\Sigma}^2)\right]=\mathbb{E}\left[ \bm{\lambda}^\textrm{true}\right]^2+\mathbb{V}\left[\bm{\lambda}^\textrm{true}\right]
\end{equation}
where $\mathbb{E}\left[ \bm{\lambda}^\textrm{true}\right]=1$ if we consider correlation matrices, and
\begin{equation}
    \mathbb{E}\left[\tau(\bm{\Lambda}^O(\bm{V}_\textrm{in},\bm{\Sigma}))^2)\right]=\mathbb{E}\left[ \bm{\lambda}^{\textrm{true}}\right]^2+\mathbb{V}\left[\bm{\lambda}^\textrm{O}\right].
\end{equation}
In fact, the oracle eigenvalue procedure does not modify the mean of the sample eigenvalue distribution $\mathbb{E}\left[ \bm{\lambda}\right]$ but shrinks the sample variance, i.e., $0\leq\mathbb{V}\left[\bm{\lambda}^\textrm{O}\right]\leq \mathbb{V}\left[\bm{\lambda}\right]$.
From the last two relationship, equation~\eqref{eq:after_wick} can be written as

\begin{equation}
    \mathbb{E}\left[\tau((\bm{\Xi}^\textrm{H}-\bm{\Sigma})^2)\right]=\mathbb{V}\left[\bm{\lambda}^\textrm{true}\right] - \mathbb{V}\left[\bm{\lambda}^\textrm{O}\right] + 2 \frac{\mathbb{E}\left[ \bm{\lambda}^\textrm{true}\right]^2+\mathbb{V}\left[\bm{\lambda}^\textrm{O}\right]}{t_\textrm{out}}.
\end{equation}

\subsection{White inverse Wishart case}
If the population matrix is drawn from a white inverse Wishart, then equation~\eqref{eq:after_wick} has a closed form in the high dimensional limit.

\begin{proposition}
\label{prop_1var}
 Let us consider $\bm{\Sigma} \sim \mathcal{W}^{-1}_{np}$ a white inverse Wishart of parameters $(n,p)$, with $p$ negligible relative to $n$, and $(\bm{x}_{\textrm{i}})_{1\leq i\leq t}$ i.i.d. random vectors of $\mathbb{R}^{n}$ such that $\bm{x}_{i}\sim\mathcal{N}(0,\bm{\Sigma})$.  Let $\bm{\Xi}^\textrm{H}$ the holdout estimator associated with the index partition $\{\mathcal{I}_\textrm{in},\mathcal{I}_\textrm{out}\}$ defined from $t_\textrm{out}$ integer between $1$ and $t-1$. The expected Frobenius error between the holdout estimator and the population covariance $\bm{\Sigma}$ is

\begin{equation}
\label{eq:result_error_holdout_asymptotics}
\mathbb{E}\left[\tau((\bm{\Xi}^\textrm{H}-\bm{\Sigma})^2)\right]=\left(\frac{2k}{t}-1\right) \left(\frac{p^2}{p+\frac{kn}{kt-t}}+1\right)+1+p.
\end{equation}
\end{proposition}

To prove it, we first note that when the population matrix is a white inverse Wishart and the data are Gaussian, the  oracle estimator recovers the optimal linear shrinkage in the Bayesian sense, in the high dimension limit when $p$ is negligible compared to $n$ and is equal to \citep{Potters2020AFC}
\begin{equation}
\label{eq:linear_shrink}
\bm{\Xi}^{\textrm{O}}=r\bm{E}+(1-r)\bm{\mathds{1}}, \hspace{4mm} \mbox{with:} \hspace{4mm} r=\frac{p}{p+q}.
\end{equation}

\begin{proof}

Let us apply the linear shrinkage formula in equation~\eqref{eq:linear_shrink} to the oracle term appearing in equation~\eqref{eq:after_wick}
\begin{equation}
   \mathbb{E}\left[\tau(\diag(\bm{V}_\textrm{in}^{T}\bm{\Sigma} \bm{V}_\textrm{in})^2)\right] =\mathbb{E}\left[(\tau (r_\textrm{in}\bm{E}_\textrm{in}+(1-r_\textrm{in})\bm{\mathds{1}})^2) \right],
\end{equation}

where $r_\textrm{in}$ is obtained by using $q_\textrm{in}=n/t_\textrm{in}$ instead of $q$. The former equation can then be expanded in the following way
\begin{equation}
   \label{eq:application_ls_formula}
   \mathbb{E}\left[\tau(\diag(\bm{V}_\textrm{in}^{T}\bm{\Sigma} \bm{V}_\textrm{in})^2)\right] =\mathbb{E}\left[r_\textrm{in}^2\tau(\bm{E}_{in}^2) + 2r_\textrm{in}(1-r_\textrm{in})\tau(\bm{E}_\textrm{in}) + (1-r_\textrm{in})^2\right].
\end{equation}

Since the data are Gaussian, one can combine equations \eqref{eq:sigma_square} and ~\eqref{eq:trace_samplecov_square} if $p$ is negligible relative to $n$, then

\begin{equation}\label{eq:critialp}
    \mathbb{E}\left[\tau(\bm{E}_\textrm{in}^2) \right] = 1+p+ q_\textrm{in},
\end{equation}

and being the population matrix a white inverse Wishart,  from the definition~\eqref{def:white_inv} and Eq.\eqref{eq:exp_trace_invW}, we have 
\begin{equation}
    \mathbb{E}\left[\tau(\bm{E}_\textrm{in}) \right] = 1.
\end{equation}

Therefore equation~\eqref{eq:application_ls_formula} becomes

\begin{equation}
   \label{eq:oracle_invW_ls}\mathbb{E}\left[\tau(\diag(\bm{V}_\textrm{in}^{T}\bm{\Sigma} \bm{V}_\textrm{in})^2)\right] =r_\textrm{in}^2\left(p+q_\textrm{in}\right) + 1. 
\end{equation}

By substituting the linear shrinkage formula of equation~\eqref{eq:oracle_invW_ls} into the expression of the holdout Frobenius error of equation~\eqref{eq:after_wick} we obtain 

\begin{equation}
\mathbb{E}\left[\tau(\bm{\Xi}^\textrm{H}-\bm{\Sigma})^2)\right]=\left(\frac2{t_\textrm{out}}-1\right)\left[r_\textrm{in}^2\left(p+q_\textrm{in}\right) + 1\right]+ 1+p,
\end{equation}

Expressing the former expression as a function of $k$ we obtain the expected error of the holdout estimator for the white inverse Wishart case in the high-dimension limit

\begin{equation}
\label{eq:holdout_error_p_negligible}
\mathbb{E}\left[\tau(\bm{\Xi}^\textrm{H}-\bm{\Sigma})^2)\right]=\left(\frac{2k}{t}-1\right) \left(\frac{p^2}{p+\frac{kn}{kt-t}}+1\right)+1+p.    
\end{equation}
However, the former equation, requires that $p$ is negligible with respect to $n$ for the equation~\eqref{eq:critialp} to be valid.
\end{proof}

In the left panel of Fig.~\ref{fig:expected_vs_realized_holdout_error}, we compared the theoretical error formula Eq.~\eqref{eq:result_error_holdout_asymptotics} with extensive Monte Carlo estimations. The theoretical error shows a good agreement with the data when $p$ is negligible compared to $n$. However, when $p$ is larger compared to $n$, the error formula shows a systematic underestimation of the error. This behavior appears clear since only the points associated with $p/n>10^{-2}$ are biased. The reason is probably that both in the optimal linear shrinkage and in our derivation, equation~\ref{eq:tsq}, requires $p$ significantly smaller than $n$. 

In the right panel of Fig.~\ref{fig:expected_vs_realized_holdout_error}, we show the theoretical and Monte Carlo error as a function of $k$ for a specific set of $\{n, p, q\}$. Interestingly, even though for our choice of parameters, a small bias is still observable, the position of the minimum seems not affected. This motivates us to use Eq.~\eqref{eq:result_error_holdout_asymptotics} to derive an analytical expression for the optimal $k$.

\begin{figure}[tbh]
        \includegraphics[width=0.495\columnwidth]{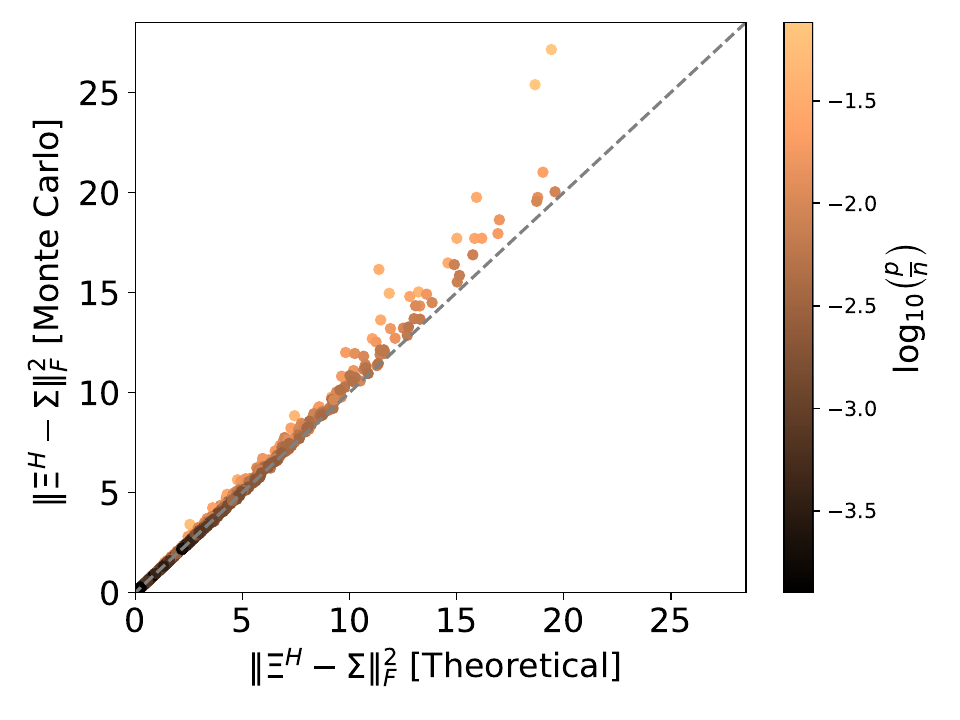}  
        \includegraphics[width=0.495\columnwidth]{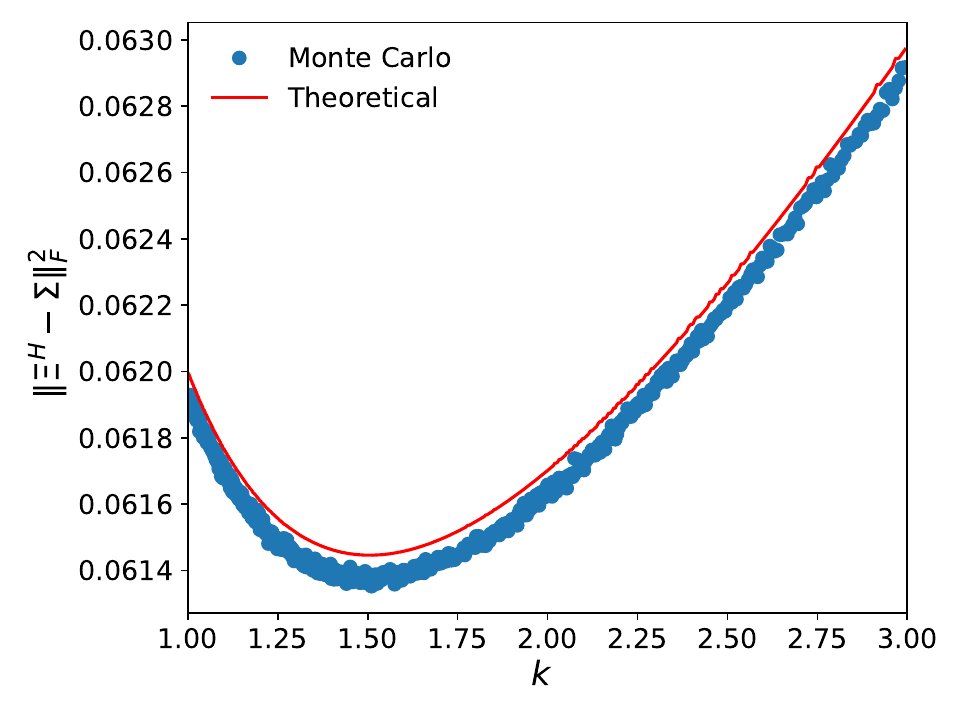}
    \caption{The left panel shows the comparison of the theoretical error of Eq.~\eqref{eq:result_error_holdout_asymptotics} with a Monte Carlo estimations for a random uniform selection of the parameters: $n \in [ 100,1000 ]$, $p\in [0.1,9 ]$, $q\in [0.1,0.9]$ and $k$ among the divisors of $t$. For stability, every combination of parameters is repeated  $100$ times, and the estimated error is averaged. The right panel shows the theoretical Frobenius error and the Monte Carlo estimation as a function of $k$ for $n=750$, $t=1000$, and $p=0.06$  averaged $1000$ times.} 
    \label{fig:expected_vs_realized_holdout_error}
\end{figure} 

\subsubsection{Optimal split}
One interesting consequence of the expression of the holdout error is the computation of the optimal number $k$ that minimizes it.

\begin{corollary} 
    Under the assumptions of proposition \eqref{prop_1var}, the $k$ that minimizes the Frobenius error of estimation is
   \begin{equation}
    \label{eq:k_opt}
    k_{\textrm{opt}} =\frac{p \left(2q + p \left(2 + 2p + q\right) + \sqrt{p^2 q^2 + 2n(p + q)(p + p^2 + q)}\right)}{2(p + q)(p + p^2 + q)}.
    \end{equation}
\end{corollary}

\begin{proof}
This expression is directly found by setting the first derivative of the Frobenius error with respect to $k$ to zero and solving a second-order polynomial.
\end{proof}

The leading term of the asymptotic expansion of Eq.~\eqref{eq:k_opt}, for $n \to \infty$ simplifies to

\begin{equation}
k_{\textrm{opt}}\sim\frac{p}{\sqrt{2(p + q)(p + p^2 + q)}} \sqrt{n}.
\end{equation}

Contrary to the standard CV practice of using a fixed train-test ratio, our analysis proves that the optimal holdout split scales proportionally with the square root of the number of features $n$ in the high-dimensional limit.

Then, we quantify the performances of the holdout estimator by comparing its expected Frobenius error with the one of the oracle estimator. We recall that the oracle estimator is the RIE that minimizes by construction the Frobenius error.  For a white inverse Wishart population matrix $\bm{\Sigma}$ when  $p$ is negligible compared to $n$ and $1\ll k\ll n$, the two errors coincide in the high dimensional limit.

\begin{corollary}
    Under the assumptions of proposition \ref{prop_1var}, if $1\ll k\ll n$, then the holdout estimator recovers the expected Frobenius error of the oracle estimator in the high dimension limit

    \begin{equation}
    \lim\limits_{n \to \infty}\mathbb{E}\left[\tau((\bm{\Xi}^{\textrm{H}}-\bm{\Sigma})^2)\right] = \lim\limits_{n \to \infty} \mathbb{E}\left[\tau(\bm{\Xi}^{\textrm{\textrm{O}}}-\bm{\Sigma})^2)\right].
    \end{equation}
    
\end{corollary}

\begin{proof}
    By assuming that $1\ll k\ll n$ in the high-dimensional limit we have the following limits 
    \begin{equation}
        \frac{2k}{t}-1 \to -1,
    \quad\textrm{and }\quad
        \frac{p^{2}}{p+\frac{nk}{kt-1}} + 1\to \frac{p^{2}}{p+q} + 1.
    \end{equation}

Then the holdout error formula of Eq.~\eqref{eq:result_error_holdout_asymptotics} converges to

\begin{align*}
     \left(\frac{2k}{t}-1\right) \left(\frac{p^2}{p+\frac{kn}{kt-t}}+1\right)+1+p \to -\left(\frac{p^{2}}{p+q} +1 \right) + 1 + p = \frac{pq}{p+q},
\end{align*}

which coincides with Eq.~\eqref{er_lp}, and proves the convergence to the expected Frobenius error of the oracle estimator.
\end{proof}

 This result was expected considering \cite{Lam2016,lam2015supplement} convergence result written in therorem \ref{theorem_lam}, but in our setting we can be a bit less restrictive about the choice of the split in the high-dimensional regime. For example for a $k$ proportional to $n^{2/3}$, we would obtain a $t_{\textrm{out}}$ proportional to $n^{1/3}$ which would not verify the convergence assumption of Lam but still verifies the assumptions of the previous corollary.

However, it is important to notice that the holdout method is defined only for $k > 1$. For $k >2$ the test set is smaller than the train set, which is the typical use in the applications. When $k$ is an integer $\{2, 3,.., t\}$, there is a direct mapping between the train-test factor ratio of the holdout method and the number of folds of a $k$-fold CV. 

It is worth remarking that the minimum in the right panel of Fig.~\ref{fig:expected_vs_realized_holdout_error}, around $k=1.5$, is a prediction confirmed by our formula written in Eq.~\eqref{eq:k_opt}. Such a minimum is counterintuitive since for $k=1.5$, the prediction of the formula implies a test set double the train set.

\section{Conclusion}

In this work, we computed analytically the expected Frobenius error of the holdout estimator for Gaussian stationary data. We obtained a closed-form expression for the case of white inverse Wishart population matrices, while for the general case numerical approximations are needed. Our approach involved the use of Wick's theorem to link the Frobenius error of estimation with the moments of the population covariance matrix and then the recognition of the oracle eigenvalues. Furthermore, the method developed by Ledoit and P\'ech\'e for computing the oracle eigenvalues can be applied to a wider range of population covariance matrix distributions beyond the inverse Wishart case. This suggests that our approach could be extended to derive analytical formulas for the holdout error in these more general settings, opening up interesting avenues for future research.

With our approach, we found that the optimal split that minimizes the estimation error is proportional to the square root of the matrix dimension. The appearance here of a power law should be further investigated in future works as they often hold significance for applications and one can wonder whether a similar pattern would remain valid for other noise models and population covariance distributions. In the finite but large limit, the error sharpens around one optimal split instead of the plateau obtained in the asymptotic regime, making the choice of the split particularly important. This paves the way for further theoretical studies that could in the future help practitioners calibrate the holdout or CV estimators in the finite-sample case instead of fitting the split empirically directly on the data.

In addition, our findings indicate that, in general, the $k$-fold CV outperforms the holdout; however, when the order of the matrix tends to infinity, the holdout estimator converges to the oracle estimator, and $k$-fold CV might become irrelevant, as a single holdout split can achieve approximately the optimal performance of the oracle estimator. This allows for the prevention of data leakage from future eigenvectors to past eigenvectors, a safeguard that is impossible to implement with $k$-fold CV but naturally achievable with the holdout method. We believe this feature is particularly relevant in fields like finance, where preventing data leakage from future information is crucial for reliable model validation.

\subsection*{Acknowledgments}
We would like to thank Jean-Philippe Bouchaud, Damien Challet and Sandrine P\'ech\'e for particularly useful discussions and advice which helped improve the paper.

\bibliographystyle{plainnat}
\bibliography{references.bib}

\end{document}